\documentclass[10pt, reqno]{amsart}

\usepackage{amssymb,stmaryrd,mathrsfs,dsfont,amsmath}

\theoremstyle{plain}
\newtheorem{theorem}{Theorem}[section]
\newtheorem{lemma}[theorem]{Lemma}
\newtheorem{proposition}[theorem]{Proposition}
\newtheorem{corollary}[theorem]{Corollary}

\theoremstyle{definition}

\newtheorem{example}[theorem]{Example}
\newtheorem{definition}[theorem]{Definition}

\theoremstyle{remark}

\DeclareMathOperator{\reg}{reg}
\DeclareMathOperator{\ureg}{ureg}

\DeclareMathOperator{\dom}{dom}
\DeclareMathOperator{\codom}{codom}

\input xy
\xyoption{all}

\begin{document}

\title[semigroups of transformations with an invariant set]
{On certain semigroups of transformations with an invariant set}

\author[Mosarof Sarkar]{Mosarof Sarkar}
\address{Department of Mathematics, Central University of South Bihar, Gaya, India}
\email{mosarofsarkar@cusb.ac.in}

\author[Shubh N. Singh]{Shubh N. Singh}
\address{Department of Mathematics, Central University of South Bihar, Gaya, India}
\email{shubh@cub.ac.in}


\begin{abstract}

Let $X$ be a nonempty set and let $T(X)$ be the full transformation semigroup on $X$. 
The main objective of this paper is to study the subsemigroup $\overline{\Omega}(X, Y)$ of $T(X)$ defined by
\[\overline{\Omega}(X, Y) = \{f\in T(X)\colon Yf = Y\},\]
where $Y$ is a fixed nonempty subset of $X$. We describe regular elements in $\overline{\Omega}(X, Y)$ and show that $\overline{\Omega}(X, Y)$ is regular if and only if $Y$ is finite. We characterize unit-regular elements in $\overline{\Omega}(X, Y)$ and prove that $\overline{\Omega}(X, Y)$ is unit-regular if and only if $X$ is finite. We characterize Green's relations on $\overline{\Omega}(X, Y)$ and prove that $\mathcal{D} =\mathcal{J}$ on $\overline{\Omega}(X, Y)$ if and only if $Y$ is finite. We also determine ideals of $\overline{\Omega}(X, Y)$ and investigate its kernel. This paper extends several results appeared in the literature.

\end{abstract}



\maketitle
\section{Introduction}

Throughout this paper, let $X$ be a nonempty set and let $Y$ be a fixed nonempty subset of $X$. Denote by $T(X)$ the full transformation semigroup on $X$. It is well-known that $T(X)$ is regular \cite[p. 63, Exercise 15]{howie95}, and $T(X)$ is unit-regular if and only if $X$ is finite \cite[Proposition 5]{alar-s80}. A characterization of Green's relations on $T(X)$ is well-known \cite[p. 63, Exercise 16]{howie95}. In $1952$, Mal'cev \cite{mal-52} determined ideals of $T(X)$. A \emph{transformation semigroup} is a subsemigroup of $T(X)$. Various transformation semigroups have been introduced and extensively studied for several decades, as every semigroup is isomorphic to a transformation semigroup \cite[Theorem 1.1.2]{howie95}.

\vspace{0.1cm}

In $1966$, Magill \cite{magill66} introduced and studied an interesting transformation semigroup $\overline{T}(X, Y)$ defined by
\[\overline{T}(X, Y) = \{f \in T(X)\colon Yf \subseteq Y\},\] where $Yf$ is the image of $Y$ under $f$. In $2005$, Nenthein et al.  proved that $\overline{T}(X, Y)$ is regular if and only if $X = Y$ or $|Y| = 1$ \cite[Corollary 2.4]{nenth05}, and they characterized regular elements in $\overline{T}(X, Y)$ \cite[Theorem 2.3]{nenth05}. Honyam and Sanwong \cite{hony11} determined Green's relations on $\overline{T}(X, Y)$ and characterized its ideals. Sarkar and Singh gave a characterization of unit-regular elements in $\overline{T}(X, Y)$ \cite[Theorem 4.2]{shubh-lma21} and also proved that $\overline{T}(X, Y)$ is unit-regular if and only if $X$ is finite, and $|Y| = 1$ or $Y = X$ \cite[Theorem 4.4]{shubh-lma21}.

\vspace{0.1cm}

Denote by $S(X)$ (resp. $I_X$) the symmetric group (resp. identity map) on $X$. Honyam and Sanwong \cite{hony-qm13} and Laysirikul \cite{laysi16} introduced, respectively, the transformation semigroups $\text{Fix}(X, Y)$ and $\overline{S}(X, Y)$ defined by
\[\text{Fix}(X, Y) = \{f \in T(X)\colon f_{\upharpoonright_Y} = I_Y\} \hspace{0.3cm}\text{ and } \hspace{0.3cm}
\overline{S}(X, Y) = \{f \in T(X)\colon f_{\upharpoonright_Y}\in S(Y)\},\]
where $f_{\upharpoonright_Y}\colon Y \to Y$ is the map that agrees with $f$ on $Y$.
Honyam and Sanwong \cite{hony-qm13} observed that $\text{Fix}(X, Y)$ is regular and characterized its Green's relations and ideals.  Chaiya et al. \cite[Theorem 5.2]{chaiya-t17} proved that $\mbox{ Fix}(X,Y)$ is unit-regular if and only if $X\setminus Y$ is finite. Laysirikul \cite[Theorem 2.2]{laysi16} showed that $\overline{S}(X, Y)$ is regular. Sommanee \cite{somma-tjm21} gave a description of Green's relations on $\overline{S}(X, Y)$ and characterize its ideals.

\vspace{0.1cm}
Denote by $\Omega(X)$ the semigroup under composition of all surjective transformations on $X$. In \cite{koni-ac19}, Konieczny characterized Green's relations on $\Omega(X)$ and investigated its kernel. In particular, Konieczny \cite{koni-ac19} proved that the kernel of $\Omega(X)$ exists and determined its elements and cardinality.


\vspace{0.1cm}
Let
\[\overline{\Omega}(X,Y)=\{f\in T(X)\colon Yf = Y\}.\]
Clearly $\overline{\Omega}(X,Y)$ is a subsemigroup of $\overline{T}(X, Y)$. If $Y = X$, then $\overline{\Omega}(X,Y) = \Omega(X)$. Therefore we may regard $\overline{\Omega}(X,Y)$ as a generalization of $\Omega(X)$. Moreover, if $Y$ is finite, then $\overline{S}(X,Y) = \overline{\Omega}(X,Y)$ (cf. \cite[Proposition 1.1.3]{gan-maz09}). By definition of $\overline{T}(X,Y)$, $\overline{S}(X,Y)$, $\text{Fix}(X,Y)$, and $\overline{\Omega}(X,Y)$, we can see that 
\[\text{Fix}(X,Y) \subseteq \overline{S}(X,Y) \subseteq  \overline{\Omega}(X,Y) \subseteq \overline{T}(X,Y).\]


\vspace{0.1cm}
The main purpose of this paper is to study the semigroup $\overline{\Omega}(X, Y)$. We divide the paper into five sections.
In Section $2$, we present basic definitions and introduce remaining notation used within the paper. In Section $3$, we first describe regular elements in $\overline{\Omega}(X, Y)$ and then show that $\overline{\Omega}(X, Y)$ is regular if and only if $Y$ is finite. Next, we characterize unit-regular elements in both semigroups $\overline{S}(X, Y)$ and $\overline{\Omega}(X, Y)$, and show that the sets of all unit-regular elements in $\overline{S}(X, Y)$ and $\overline{\Omega}(X, Y)$ are equal. 
Also, we prove that $\overline{S}(X, Y)$ (resp. $\overline{\Omega}(X, Y)$) is unit-regular if and only if $X\setminus Y$ (resp. $X$) is finite. In Section $4$, we characterize Green's relations on $\overline{\Omega}(X, Y)$ and prove that $\mathcal{D} = \mathcal{J}$ on $\overline{\Omega}(X, Y)$ if and only if $Y$ is finite. 
In Section $5$, we characterize ideals of $\overline{\Omega}(X, Y)$ and determine its kernel.


\section{Preliminaries and Notation}

Let $A$ be a set. We denote by $|A|$ the cardinality of $A$ and write $A\setminus B = \{x\in A\colon x\notin B\}$, where $B$ is a set. A \emph{partition} of $A$ is a collection of pairwise disjoint nonempty subsets of $A$, called \emph{blocks}, whose union is $A$. A \emph{transversal} of an equivalence relation $\rho$ on $A$ is a subset of $A$ which contains exactly one element of each $\rho$-class. We denote by $I_A$ the identity map on $A$.

\vspace{0.1cm}

We denote the composition of maps by juxtaposition, and we compose maps from left to right. Let $f$ be a map. We write $\dom(f)$ (resp. $\codom(f)$) to denote the domain (resp. codomain) of $f$. For $x\in \dom(f)$, we write $xf$ for the image of $x$ under $f$. If $B \subseteq \dom(f)$ and $C\subseteq \codom(f)$, we let $Bf = \{xf \colon x\in B\}$ and $Cf^{-1} = \{x\in \dom(f)\colon xf \in C\}$. We simply write $xf^{-1}$ instead of $Cf^{-1}$ when $C = \{x\}\subseteq \codom(f)$. Let $f$ be a map and $\dom(f) = A$. We denote by $\ker(f)$ the equivalence relation on $A$ defined by $(a,b)\in \ker(f)$ if and only if $af = bf$. We write $T_{f}$ to denote a transversal of $\ker(f)$. Note that $|T_f| = |Af|$.  Let $f\colon A \to A$ be a map and $B \subseteq \dom(f)$. The \emph{restriction} of $f$ to $B$ is a map $f_{|_B} \colon B \to A$ defined by $x(f_{|_B}) = xf$ for all $x \in B$. If $Bf \subseteq B$, we denote by $f_{\upharpoonright_B}$ the map from $B$ to $B$ which agrees with $f$ on $B$. Note that if $f\in \overline{\Omega}(X,Y)$ (resp. $f\in \overline{S}(X,Y)$), then $f_{\upharpoonright_Y}\in \Omega(Y)$ (resp. $f_{\upharpoonright_Y}\in S(Y)$).

\vspace{0.1cm}

Let $S$ be a semigroup with identity. Denote by $U(S)$ the set of all units in $S$.
An element $x\in S$ is called \emph{regular} (resp. \emph{unit-regular}) if $xyx = x$ for some $y\in S$ (resp. $y\in U(S)$). Let $\reg(S)$ (resp. $\ureg(S)$) denote the set of all regular (resp. unit-regular) elements in $S$. The semigroup $S$ is said to be \emph{regular} (resp. \emph{unit-regular}) if $\reg(S) = S$ (resp. $\ureg(S) = S$). 
Note that $U(T(X)) = S(X)$, and if $f\in U(\overline{T}(X,Y))$, then $f_{\upharpoonright_Y}\in S(Y)$.

\vspace{0.1cm}

Let $S$ be a semigroup. As usual, we denote Green's relations on $S$ by $\mathcal{L}$, $\mathcal{R}$, $\mathcal{H}$, $\mathcal{D}$, and $\mathcal{J}$. If $a\in S$ and $\mathcal{T}$ is one of Green's relations on $S$, then we denote by $T_a$ the equivalence class of $a$ with respect to $\mathcal{T}$. An ideal $I$ of $S$ is said to be \emph{minimal} if there is no ideal of $S$ that is strictly contained in $I$. Note that a semigroup contains at most one minimal ideal. However, if a semigroup $S$ contains a minimal ideal, then such unique minimal ideal of $S$ is called the \emph{kernel} of $S$ and is denoted by $K(S)$.

\vspace{0.1cm}
For all undefined notions and notation, we refer the reader to \cite{clifford61, howie95}.

\section{Regular and Unit-regular elements}
 
In this section, we first describe regular elements in $\overline{\Omega}(X, Y)$ and show that $\overline{\Omega}(X, Y)$ is regular if and only if $Y$ is finite. Next, we characterize unit-regular elements in both semigroups $\overline{S}(X, Y)$ and $\overline{\Omega}(X, Y)$, and prove that the sets of all unit-regular elements in $\overline{S}(X, Y)$ and $\overline{\Omega}(X, Y)$ are equal. We prove that $\overline{S}(X, Y)$  (resp. $\overline{\Omega}(X, Y)$) is unit-regular if and only if $X\setminus Y$ (resp. $X$ ) is finite. Finally, we give new proofs of known results about regularity and unit-regularity of $\mbox{ Fix}(X,Y)$.

\vspace{0.1cm}
The following simple lemma is included here for completeness.
\vspace{0.1cm}
\begin{lemma}\label{trans-image}
	Let $f\colon X \to X$ be a map and let $T_f$ be a transversal of $\ker(f)$. Then
	
	\begin{enumerate}
		\item[\rm(i)] $f$ is injective if and only if  $X\setminus T_f = \emptyset$;
		\item[\rm(ii)] $f$ is surjective if and only if  $X\setminus Xf = \emptyset$;
		\item[\rm(iii)] $f$ is bijective if and only if $X\setminus T_f =\emptyset$ and $X\setminus Xf = \emptyset$; 
		\item[\rm(iv)] if $f$ is injective but not surjective, then $|X\setminus T_f| \neq |X\setminus Xf|$. 
	\end{enumerate}
\end{lemma}

\vspace{0.1cm}

In general, the subset $\reg(S)$ of a semigroup $S$ does not form a subsemigroup of $S$.
However, in the next result, we show that $\reg(\overline{\Omega}(X,Y))$ is a subsemigroup of $\overline{\Omega}(X,Y)$.

\begin{theorem}\label{regular-OXY}
We have $\reg(\overline{\Omega}(X,Y))=\overline{S}(X,Y)$.
\end{theorem}

\begin{proof}
	Note that $\overline{S}(X,Y)$ is a regular semigroup \cite[Theorem 2.2]{laysi16}. Since $\overline{S}(X,Y)\subseteq \overline{\Omega}(X,Y)$, it follows that $\overline{S}(X,Y)= \reg(\overline{S}(X,Y))\subseteq \reg(\overline{\Omega}(X,Y))$. To show the reverse inclusion, let $f\in \reg(\overline{\Omega}(X,Y))$. Then there exists $g\in \overline{\Omega}(X,Y)$ such that $fgf=f$. This gives $f_{\upharpoonright_Y}g_{\upharpoonright_Y}f_{\upharpoonright_Y}=f_{\upharpoonright_Y}$ and so $Y(f_{\upharpoonright_Y}g_{\upharpoonright_Y})$ is a transversal of $\ker(f_{\upharpoonright_Y})$ by \cite[Lemma 3.1]{shubh-s1-21}. Since $f_{\upharpoonright_Y}, g_{\upharpoonright_Y} \in \Omega(Y)$, it is clear that $Y(f_{\upharpoonright_Y}g_{\upharpoonright_Y})=Y$. Since $Y(f_{\upharpoonright_Y}g_{\upharpoonright_Y})$ is a transversal of $\ker(f_{\upharpoonright_Y})$ and $Y(f_{\upharpoonright_Y}g_{\upharpoonright_Y})=Y$, it follows that
	$Y\setminus Y(f_{\upharpoonright_Y}g_{\upharpoonright_Y})= \emptyset$ and so $f_{\upharpoonright_Y}$ is injective by Lemma \ref{trans-image}(i). Hence $f_{\upharpoonright_Y}\in S(Y)$ and subsequently $f\in \overline{S}(X,Y)$ as required.
\end{proof}

\vspace{0.1cm}
Note that $\overline{S}(X,Y) =\overline{\Omega}(X,Y)$ if and only if $Y$ is finite (cf. \cite[Proposition 1.1.3]{gan-maz09}). Using \cite[Theorem 2.2]{laysi16} together with Theorem \ref{regular-OXY}, we obtain the following result.

\begin{proposition}\label{reg-OXY}
The semigroup $\overline{\Omega}(X,Y)$ is regular if and only if $Y$ is finite.
\end{proposition}

In the next result, we characterize unit-regular elements in $\overline{S}(X,Y)$.

\begin{theorem}\label{ureg-element-SXY}
Let $f\in \overline{S}(X,Y)$. Then $f\in \ureg (\overline{S}(X,Y))$ if and only if  $|X\setminus T_f|=|X\setminus Xf|$ for some transversal $T_f$ of $\ker(f)$ such that $Y\subseteq T_f$.
\end{theorem}

\begin{proof}
Suppose that $f\in \ureg(\overline{S}(X,Y))$. Since $\ureg(\overline{S}(X,Y)) \subseteq \ureg(\overline{T}(X,Y))$, we have $f\in \ureg(\overline{T}(X,Y))$. Then, by \cite[Theorem 4.2]{shubh-lma21}, there exists a transversal $T_f$ of $\ker(f)$ such that $ Y\cap T_f$ is a transversal of $\ker(f_{\upharpoonright_Y})$ and 
\begin{equation}\label{size-trans-diff}
|(X\setminus T_f)\setminus (Y\setminus T_{f_{\upharpoonright_Y}})| = |(X\setminus Xf)\setminus (Y\setminus Yf_{\upharpoonright_Y})| 
\end{equation}
where $T_{f_{\upharpoonright_Y}} = Y\cap T_f$ is a transversal of $\ker(f_{\upharpoonright_Y})$. Since $f_{\upharpoonright_Y} \in S(Y)$, the set $Y$ is the only transversal of $\ker(f_{\upharpoonright_Y})$. This implies that $Y = Y\cap T_f$ and so $Y \subseteq T_f$. Further, since $f_{\upharpoonright_Y} \in S(Y)$, it follows from Lemma \ref{trans-image}(iii) that
$Y\setminus T_{f_{\upharpoonright_Y}} = \emptyset$ and $Y\setminus Yf_{\upharpoonright_Y} = \emptyset$. Then, by using  equation (\ref{size-trans-diff}), we obtain \[|X\setminus T_f| = |(X\setminus T_f)\setminus (Y\setminus T_{f_{\upharpoonright_Y}})| = |(X\setminus Xf)\setminus (Y\setminus Yf_{\upharpoonright_Y})| = |X\setminus Xf|.\] 
 
Conversely, suppose that $|X\setminus T_f|=|X\setminus Xf|$ for some transversal $T_f$ of $\ker(f)$ such that $Y\subseteq T_f$. Since $f_{\upharpoonright_Y} \in S(Y)$, the set $Y$ is the only transversal of $\ker(f_{\upharpoonright_Y})$. Write $T_{f_{\upharpoonright_Y}} = Y$. By hypothesis, since $Y\subseteq T_f$, we then get $T_{f_{\upharpoonright_Y}} = Y \cap T_f$. Further, since $f_{\upharpoonright_Y} \in S(Y)$, it follows from Lemma \ref{trans-image}(iii) that $Y\setminus T_{f_{\upharpoonright_Y}} = \emptyset$ and $Y\setminus Y f_{\upharpoonright_Y} = \emptyset$. Therefore, by hypothesis, we obtain
\[|(X\setminus T_f)\setminus (Y\setminus T_{f_{\upharpoonright_Y}})| =  |X\setminus T_f| = |X\setminus Xf| = |(X\setminus Xf)\setminus (Y\setminus Yf_{\upharpoonright_Y})|.\]
In addition, since $f_{\upharpoonright_Y}\in S(Y)$, we simply have $f_{\upharpoonright_Y}\in \ureg(T(Y))$ and $Yf_{\upharpoonright_Y} = Y \cap Xf$. 
Therefore, by \cite[Theorem 4.2]{shubh-lma21}, we get $f\in \ureg(\overline{T}(X,Y))$.
Since $U(\overline{T}(X,Y))\subseteq \overline{S}(X,Y)$, we conclude that $f\in \ureg(\overline{S}(X,Y))$.
\end{proof}

\vspace{0.1cm}
In the next result, we show that $\ureg(\overline{S}(X,Y))$ and $\ureg(\overline{\Omega}(X,Y))$ are equal. 

\begin{theorem}\label{ureg-OXY}
We have	$\ureg(\overline{\Omega}(X,Y))=\ureg(\overline{S}(X,Y))$.
\end{theorem}

\begin{proof}
	Since $\overline{S}(X,Y)\subseteq \overline{\Omega}(X,Y)$, it follows that $\ureg(\overline{S}(X,Y))\subseteq \ureg(\overline{\Omega}(X,Y))$. To show the reverse inclusion, let $f\in \ureg(\overline{\Omega}(X,Y))$. Then $f\in \reg(\overline{\Omega}(X,Y))$ and so $f\in \overline{S}(X,Y)$ by Theorem \ref{regular-OXY}. Since $f\in \ureg(\overline{\Omega}(X,Y))$, there exists $g\in U(\overline{\Omega}(X,Y))$ such that $fgf = f$. Note that $U(\overline{\Omega}(X,Y))\subseteq U(\overline{S}(X,Y))$. Therefore $g\in U(\overline{S}(X,Y))$ and so $f\in \ureg(\overline{S}(X,Y))$ as required.
\end{proof}

\vspace{0.1cm}
Using Theorems \ref{ureg-element-SXY} and \ref{ureg-OXY}, we obtain the following corollary. 
\begin{corollary}\label{ureg-element-OXY}
Let $f\in \overline{\Omega}(X,Y)$. Then $f\in \ureg (\overline{\Omega}(X,Y))$ if and only if  $|X\setminus T_f|=|X\setminus Xf|$ for some transversal $T_f$ of $\ker(f)$ such that $Y\subseteq T_f$.
\end{corollary}

\vspace{0.1cm}

\begin{lemma}\label{ureg-l-FXY}
If $X\setminus Y$ is an infinite set, then there exists $f\in \text{Fix}(X, Y)$ such that $f\notin \ureg(\overline{S}(X,Y))$.
\end{lemma}

\begin{proof}
If $X\setminus Y$ is an infinite set, then there exists a map $\alpha \colon X\setminus Y\to X\setminus Y$ which is injective but not surjective. Define $f\colon X\to X$ by 	
	\begin{align*}
	xf=
	\begin{cases}
		x & \text{ if $x\in Y$},\\
		x\alpha & \text{ if $x\in X\setminus Y$.}
	\end{cases}
\end{align*}
It is obvious that $f\in \text{Fix}(X, Y)$. Since $\alpha$ is injective, it follows that $f$ is injective and so $X$ is the only transversal of $\ker(f)$. Write $T_f = X$. Then $Y\subseteq T_f$ and $|X\setminus T_f|= 0$. In addition, since $\alpha$ is not surjective, the map $f$ is not surjective and so $|X\setminus Xf|\ge 1$. This implies that $|X\setminus T_f|\neq |X\setminus Xf|$ where $T_f$ is the only transversal of $\ker(f)$ and $Y\subseteq T_f$. Hence $f\notin \ureg(\overline{S}(X,Y))$ by Theorem \ref{ureg-element-SXY}. This completes the proof of the lemma.
\end{proof}

\vspace{0.1cm}
Our next result characterizes unit-regularity of $\overline{S}(X,Y)$.
\begin{theorem}\label{ureg-SXY}
	The semigroup $\overline{S}(X,Y)$ is unit-regular if and only if $X\setminus Y$ is finite.
\end{theorem}

\begin{proof}
	Suppose that $\overline{S}(X,Y)$ is unit-regular. On the contrary, assume that $X\setminus Y$ is infinite. Then, by Lemma \ref{ureg-l-FXY}, there exists $f\in \text{Fix}(X, Y)$ such that $f\notin \ureg(\overline{S}(X,Y))$. However, since $\text{Fix}(X, Y)\subseteq \overline{S}(X,Y)$, we have $f\in \overline{S}(X,Y)$ and so $f\in \ureg(\overline{S}(X,Y))$ by hypothesis. This is a contradiction of the fact that $f\notin \ureg(\overline{S}(X,Y))$. Hence $X\setminus Y$ is finite.

\vspace{0.1cm}
Conversely, suppose that $X\setminus Y$ is finite and let $f\in \overline{S}(X,Y)$. Clearly $Y \subseteq Xf$ and therefore  $X\setminus Xf, Xf\setminus Y\subseteq X\setminus Y$. Since $f_{\upharpoonright_Y} \in S(Y)$, it is clear that $Y$ is the only transversal of $\ker(f_{\upharpoonright_Y})$ and there exists a transversal $T_f$ of $\ker(f)$ such that $Y\subseteq T_f$. Then $X\setminus T_f, T_f\setminus Y \subseteq X\setminus Y$.  Since $f_{\upharpoonright_Y} \colon Y \to Y$ and $f_{|_{T_f}}\colon T_f \to Xf$ are bijections, it follows that $|T_f\setminus Y| = |Xf\setminus Y|$. Further, since $X\setminus Y$ is finite, the sets $X\setminus Xf$, $Xf\setminus Y$, $X\setminus T_f$, and $T_f\setminus Y$ are finite. Therefore
\[|(X\setminus Y)\setminus (T_f\setminus Y)| =|X\setminus Y|- |T_f\setminus Y|=
|X\setminus Y|-|Xf\setminus Y|= |(X\setminus Y)\setminus (Xf\setminus Y)|\]
and so
$|X\setminus T_f| = |(X\setminus Y)\setminus (T_f\setminus Y)| = |(X\setminus Y)\setminus (Xf\setminus Y)| = |X\setminus Xf|$.
This implies that $f\in \ureg(\overline{S}(X,Y))$ by Theorem \ref{ureg-element-SXY}. Since $f$ is arbitrary, we conclude that $\overline{S}(X,Y)$ is unit-regular.	
\end{proof}

\vspace{0.1cm}

Using Proposition \ref{reg-OXY} and Theorem \ref{ureg-SXY}, we obtain the following.

\begin{theorem}
The semigroup $\overline{\Omega}(X,Y)$ is unit-regular if and only if $X$ is finite.
\end{theorem}
\begin{proof}
Suppose that $\overline{\Omega}(X,Y)$ is unit-regular. Then $\overline{\Omega}(X,Y)$ is regular and so $Y$ is finite by Proposition \ref{reg-OXY}. Therefore $\overline{\Omega}(X,Y)=\overline{S}(X,Y)$. Since 
$\overline{\Omega}(X,Y)$ is unit-regular and $\overline{\Omega}(X,Y)=\overline{S}(X,Y)$, it follows that $X\setminus Y$ is finite by Theorem \ref{ureg-SXY}. Moreover, since $Y$ is finite, we conclude that $X$ is finite. 
	
\vspace{0.1cm}
Conversely, suppose that $X$ is finite. Then $X\setminus Y$ is finite and so $\overline{S}(X,Y)$ is unit-regular by Theorem \ref{ureg-SXY}. Again, by hypothesis, the set $Y$ is finite and so $\overline{\Omega}(X,Y)=\overline{S}(X,Y)$. Hence $\overline{\Omega}(X,Y)$ is unit-regular.
\end{proof}

\vspace{0.1cm}
Now, we recall the notion of pre-inverses introduced by Hickey \cite[p. 372]{hick83} in 1983.
\vspace{0.1cm}
\begin{definition} \cite[p. 372]{hick83}
Let $S$ be a semigroup and $a\in S$. An element $b\in S$ is said to be a \emph{pre-inverse} of $a$ if $aba = a$. 
\end{definition}

We denote by $\text{Pre}(a)$ the set of all pre-inverses of $a$ of a semigroup $S$.

\vspace{0.1cm}
\begin{lemma}\label{lem1-FXY}
	Let $f\in \overline{T}(X,Y)$. 
	\begin{enumerate}
		\item[\rm(i)] If $f\in \overline{S}(X,Y)$, then $\text{Pre}(f) \subseteq \overline{S}(X,Y)$.
		\item[\rm(ii)] If $f\in \mbox{Fix}(X,Y)$, then $\text{Pre}(f) \subseteq \mbox{Fix}(X,Y)$.
	\end{enumerate}
\end{lemma}

\begin{proof}\
	
	\begin{enumerate}
		\item[\rm(i)] Let $g\in \text{Pre}(f)$. This means $fgf = f$. It follows that $f_{\upharpoonright_Y}g_{\upharpoonright_Y}f_{\upharpoonright_Y}=f_{\upharpoonright_Y}$. Since $f\in \overline{S}(X,Y)$, we have $f_{\upharpoonright_Y} \in S(Y)$. Therefore $f_{\upharpoonright_Y}g_{\upharpoonright_Y} = I_Y$ and $g_{\upharpoonright_Y}f_{\upharpoonright_Y} = I_Y$. Hence $g\in \overline{S}(X,Y)$ as required. 
		
		\vspace{0.1cm}
		\item[\rm(ii)] Let $g\in \text{Pre}(f)$. This means $fgf = f$. It follows that $f_{\upharpoonright_Y}g_{\upharpoonright_Y}f_{\upharpoonright_Y}=f_{\upharpoonright_Y}$. Since $f\in \mbox{Fix}(X,Y)$, we have $f_{\upharpoonright_Y} = I_Y$. Therefore $g_{\upharpoonright_Y} = I_Y$ and hence $g\in \mbox{Fix}(X,Y)$ as required.
		\end{enumerate}
\end{proof}


Next, we give a new proof of the next result which was first appeared in \cite[p. 83]{hony-qm13}.

\vspace{0.1cm}
\begin{proposition}
The semigroup \mbox{Fix}$(X,Y)$ is regular.
\end{proposition}

\begin{proof}
Let $f\in \mbox{Fix}(X,Y)$. Then $f\in \overline{S}(X,Y)$. It follows from \cite[Theorem 2.2]{laysi16} that $fgf = f$ for some $g\in \overline{S}(X,Y)$. Then, by Lemma \ref{lem1-FXY}, we have $g\in \mbox{Fix}(X,Y)$ and so $f\in \reg(\mbox{Fix}(X,Y))$. Hence $\mbox{Fix}(X,Y)$ is regular. 
\end{proof}

\vspace{0.1cm}

The following result was first proved by Chaiya \cite[Theorem 5.2]{chaiya-t17} in $2017$.

\begin{proposition}
The semigroup $\text{Fix}(X,Y)$ is unit-regular if and only if $X\setminus Y$ is finite.
\end{proposition}

\begin{proof}
Suppose that $\mbox{Fix}(X,Y)$ is unit-regular. On the contrary, assume that $X\setminus Y$ is infinite. Then, by Lemma \ref{ureg-l-FXY}, there exists $f\in \text{Fix}(X, Y)$ such that $f\notin \ureg(\overline{S}(X,Y))$.  Since $\mbox{Fix}(X,Y) \subseteq \overline{S}(X,Y)$, it follows that $f\notin \ureg(\mbox{Fix}(X,Y))$. However, since $f\in \text{Fix}(X, Y)$, we get $f\in \ureg(\mbox{Fix}(X,Y))$ by hypothesis. This is a contradiction. Hence $X\setminus Y$ is finite.

\vspace{0.1cm}
Conversely, suppose that $X\setminus Y$ is finite and let $f\in \mbox{Fix}(X,Y)$. Then $f\in \overline{S}(X,Y)$. Since $X\setminus Y$ is finite, we get $fgf = f$ for some $g\in U(\overline{S}(X,Y))$ by Theorem \ref{ureg-SXY}. Now, by Lemma \ref{lem1-FXY}, we simply have $g\in U(\mbox{Fix}(X,Y))$. Hence $f\in \ureg(\mbox{Fix}(X,Y))$ as required. 
\end{proof}

\section{Green's Relations on $\overline{\Omega}(X,Y)$}

In this section, we give a characterization of Green's relations on $\overline{\Omega}(X,Y)$ and prove that $\mathcal{D} = \mathcal{J}$ on $\overline{\Omega}(X,Y)$ if and only if $Y$ is finite. We refer the reader to \cite[Chapter 2]{howie95} for definitions and notation of Green's relations.

\vspace{0.1cm}

In order to characterize $\mathcal{L}$-relation on $\overline{\Omega}(X,Y)$, we need the following lemma. 

\begin{lemma}\label{green-l-surj}
Let $f,g\in \overline{\Omega}(X,Y)$. Then  $f=hg$ for some $h\in \overline{\Omega}(X,Y)$
 if and only if 
\begin{enumerate}
\item[\rm(i)] $Xf\subseteq Xg$;
\item[\rm(ii)] $|y(f_{\upharpoonright_Y})^{-1}|\geq |y(g_{\upharpoonright_Y})^{-1}|$ for all $y\in Y$. 
\end{enumerate}
\end{lemma}

\begin{proof}
Suppose that $f=hg$ for some $h\in \overline{\Omega}(X,Y)$. 
\begin{enumerate}
\item[\rm(i)] Since $\overline{\Omega}(X,Y) \subseteq \overline{T}(X,Y)$, we have $f, g, h \in \overline{T}(X,Y)$ and so $Xf\subseteq Xg$ by \cite[Lemma 2]{hony11}. 

\item[\rm(ii)] Note that $f_{\upharpoonright_Y}, g_{\upharpoonright_Y}, h_{\upharpoonright_Y} \in \Omega(Y)$ and $f_{\upharpoonright_Y}=h_{\upharpoonright_Y}g_{\upharpoonright_Y}$. Then, by \cite[Theorem 2.1(1)]{koni-ac19}, we get $|y(f_{\upharpoonright_Y})^{-1}|\geq |y(g_{\upharpoonright_Y})^{-1}|$ for all $y\in Y$.
\end{enumerate}	
\vspace{0.1cm}
Conversely, suppose that the given conditions hold. First, we recall (ii). Then, by \cite[Theorem 2.1(1)]{koni-ac19}, there exists $\gamma \in \Omega (Y)$ such that $f_{\upharpoonright_Y}=\gamma g_{\upharpoonright_Y}$. Next, by (i), there exists $x'\in X$ such that $xf=x'g$ for all $x\in X\setminus Y$. Therefore, for each $x\in X\setminus Y$, we may choose such an element $x'\in X$. Now, we define $h\colon X \to X$ by
	\begin{align*}
		xh=
		\begin{cases}
			x\gamma  & \mbox{ if }x\in Y,\\
			x' & \mbox{ if }x\in X\setminus Y.
		\end{cases}
	\end{align*}
It is clear that $h\in \overline{\Omega}(X,Y)$. Finally, we show that $f = hg$. Let $x\in X$. Consider two possible cases.

\vspace{0.1cm}
\noindent\textbf{Case 1:} $x\in Y$. Then 
$x(hg)=(xh)g =(x\gamma)g=(x\gamma)g_{\upharpoonright_Y}=(x(\gamma g_{\upharpoonright_Y})=xf_{\upharpoonright_Y}=xf$.

\vspace{0.1cm}
\noindent\textbf{Case 2:} $x\in X\setminus Y$. Then $x(hg)= (xh)g = x'g=xf$. 

\vspace{0.1cm}
Hence, from both cases, we conclude that $f = hg$ as required.
\end{proof}

\begin{theorem}\label{l-green}
Let $f,g\in \overline{\Omega}(X,Y)$. Then $(f, g)\in\mathcal{L}$ in  $\overline{\Omega}(X,Y)$ if and only if  $Xf = Xg$ and  $|y(f_{\upharpoonright_Y})^{-1}|= |y(g_{\upharpoonright_Y})^{-1}|$ for all $y\in Y$.
\end{theorem}

\begin{proof}
It follows immediately from Lemma \ref{green-l-surj}.
\end{proof}


\vspace{0.1cm}
\begin{definition}\label{green-r-partition}
Let $\mathcal{P}, \mathcal{Q}$ be partitions of a set $X$, and let $\mathcal{A}$ and $\mathcal{B}$ be subcollections of blocks in $\mathcal{P}$ and $\mathcal{Q}$, respectively.
We say that $\mathcal{A}$ \emph{refines} $\mathcal{B}$, denoted by $\mathcal{A} \preceq \mathcal{B}$, if for every $A \in \mathcal{A}$, there exists $B \in \mathcal{B}$ such that $A \subseteq B$. Moreover, if $\mathcal{A} \preceq \mathcal{B}$ and $\mathcal{B} \preceq \mathcal{A}$, then we write $\mathcal{A} = \mathcal{B}$.
\end{definition}

For $f\in T(X)$ and $A \subseteq \codom(f)$ with $Xf \cap A\neq \emptyset$, we let
\[\pi_f(A) = \{xf^{-1}\colon x\in Xf \cap A\}.\]
We will write $\pi(f)$ instead of $\pi_f(X)$. Note that $\pi(f)$ is a partition of $X$ induced by $f$.

\vspace{0.1cm}

To characterize $\mathcal{R}$-relation on $\overline{\Omega}(X,Y)$, we require the following lemma.

\begin{lemma}\label{green-r-surj}
Let $f,g\in \overline{\Omega}(X,Y)$. Then $f=gh$ for some $h\in \overline{\Omega}(X,Y)$ if and only if 
\begin{enumerate}
	\item[\rm(i)] $\pi(g) \preceq \pi(f)$;
	\item[\rm(ii)] $\pi_g(Y) \preceq \pi_f(Y)$. 
\end{enumerate}
\end{lemma}

\begin{proof}
Suppose that $f=gh$ for some $h\in \overline{\Omega}(X,Y)$. Clearly $f, g, h \in \overline{T}(X,Y)$. Then, by \cite[Lemma 3]{hony11}, the statements (i) and (ii) hold. 

\vspace{0.1cm}
Conversely, suppose that the given conditions hold. Clearly $f,g\in \overline{T}(X,Y)$. Then, by \cite[Lemma 3]{hony11}, there exists $h\in \overline{T}(X,Y)$ such that $f=gh$. Now, we show that $h\in \overline{\Omega}(X,Y)$. Note that $f_{\upharpoonright_Y}, g_{\upharpoonright_Y} \in \Omega(Y)$ and $f_{\upharpoonright_Y}=g_{\upharpoonright_Y}h_{\upharpoonright_Y}$. Therefore $g_{\upharpoonright_Y}h_{\upharpoonright_Y}\in \Omega(Y)$ and so $h_{\upharpoonright_Y}\in \Omega(Y)$ (cf. \cite[Proposition 2.4.6]{ramji17}). Hence  $h\in \overline{\Omega}(X,Y)$ with $f=gh$ as required.
\end{proof}

\begin{theorem}\label{r-green}
Let $f,g\in \overline{\Omega}(X,Y)$. Then $(f, g)\in \mathcal{R}$ 
in $\overline{\Omega}(X,Y)$ if and only if $\pi(f) = \pi(g)$ and $\pi_f(Y) = \pi_g(Y)$.
\end{theorem}

\begin{proof}
	It follows immediately from Lemma \ref{green-r-surj}.
\end{proof}

By definition of $\mathcal{H}$-relation, we know that $\mathcal{H} = \mathcal{L}\cap \mathcal{R}$. As an immediate consequence of Theorems \ref{l-green} and \ref{r-green}, we obtain the following.
\begin{theorem}
Let $f,g\in \overline{\Omega}(X,Y)$. Then $(f, g)\in \mathcal{H}$ in $\overline{\Omega}(X,Y)$ if and only if $Xf=Xg$, $\pi(f)=\pi(g)$, $\pi_f(Y)=\pi_g(Y)$, and $|y(f_{\upharpoonright_Y})^{-1}|=|y(g_{\upharpoonright_Y})^{-1}|$ for all $y\in Y$.	  
\end{theorem}

\begin{lemma}\label{green-betwen-restric}
	Let $f,g\in \overline{\Omega}(X,Y)$. 
	\begin{enumerate}
		\item[\rm(i)] If $(f, g)\in \mathcal{L}$ in $\overline{\Omega}(X,Y)$, then $(f_{\upharpoonright_Y}, g_{\upharpoonright_Y})\in \mathcal{L}$ in $\Omega(Y)$.
		
		\item[\rm(ii)] If $(f, g)\in \mathcal{R}$ in $\overline{\Omega}(X,Y)$, then $(f_{\upharpoonright_Y}, g_{\upharpoonright_Y})\in \mathcal{R}$ in $\Omega(Y)$.
		
		\item[\rm(iii)] If $(f, g)\in \mathcal{D}$ in $\overline{\Omega}(X,Y)$, then $(f_{\upharpoonright_Y}, g_{\upharpoonright_Y})\in \mathcal{D}$ in $\Omega(Y)$.
	\end{enumerate}
\end{lemma}

\begin{proof}\
	\begin{enumerate}
		\item[\rm(i)] If $(f, g)\in \mathcal{L}$ in $\overline{\Omega}(X,Y)$, then there exist $h, h' \in \overline{\Omega}(X,Y)$ such that $f = hg$ and $g = h'f$. Clearly
		$f_{\upharpoonright_Y}, g_{\upharpoonright_Y}, h_{\upharpoonright_Y}, h'_{\upharpoonright_Y}\in \Omega(Y)$, 
		$f_{\upharpoonright_Y} = h_{\upharpoonright_Y} g_{\upharpoonright_Y}$, and 
		$g_{\upharpoonright_Y} =  h'_{\upharpoonright_Y} f_{\upharpoonright_Y}$. Hence $(f_{\upharpoonright_Y}, g_{\upharpoonright_Y})\in \mathcal{L}$ in $\Omega(Y)$.
		
		\vspace{0.1cm}
		\item[\rm(ii)] It is a dual of the proof of (i) and so we omit the details. 
		
		\vspace{0.1cm}
		\item[\rm(iii)] Since $\mathcal{D} = \mathcal{L}\circ \mathcal{R}$, it follows immediately from (i) and (ii).
		
	\end{enumerate}
\end{proof}

In the next result, we characterize $\mathcal{D}$-relation on $\overline{\Omega}(X,Y)$.
\begin{theorem}\label{green-d-theorem}
Let $f,g\in \overline{\Omega}(X,Y)$. Then $(f, g)\in \mathcal{D}$ in $\overline{\Omega}(X,Y)$ if and only if 
\begin{enumerate}
\item[\rm(i)] $|Xf\setminus Y|=|Xg\setminus Y|$;
\item[\rm(ii)]  there exists $\alpha \in S(Y)$ such that $|y(f_{\upharpoonright_Y})^{-1}|=|(y\alpha) (g_{\upharpoonright_Y})^{-1}|$ for all $y\in Y$.
\end{enumerate}
\end{theorem}

\begin{proof}
Suppose that $(f, g)\in \mathcal{D}$ in $\overline{\Omega}(X,Y)$. 
\begin{enumerate}
\item[\rm(i)] Since $\overline{\Omega}(X,Y) \subseteq  \overline{T}(X,Y)$, it follows that $(f, g)\in \mathcal{D}$ in $\overline{T}(X,Y)$ and so $|Xf\setminus Y|=|Xg\setminus Y|$ by \cite[Theorem 4]{hony11}. 

\vspace{0.1cm}
\item[\rm(ii)] Then, by Lemma \ref{green-betwen-restric}(iii), we have $(f_{\upharpoonright_Y}, g_{\upharpoonright_Y})\in \mathcal{D}$ in $\Omega(Y)$. Hence, by \cite[Theorem 2.1(3)]{koni-ac19}, there exists $\alpha \in S(Y)$ such that $|y(f_{\upharpoonright_Y})^{-1}|=|(y\alpha) (g_{\upharpoonright_Y})^{-1}|$ for all $y\in Y$. 
\end{enumerate}	
\vspace{0.1cm}
Conversely, suppose that the given conditions hold. By (i), there exists a bijection $\beta \colon Xg\setminus Y\to Xf\setminus Y$. Now, we recall (ii). Define $h\colon X \to X$ by
\begin{align*}
	xh=
	\begin{cases}
		a_i\alpha ^{-1}  & \text{if $x\in X_i$ where $X_i\in \pi_g(Y)$ and $X_ig=a_i$},\\
		b_j\beta         & \text{if $x\in X_j$ where $X_j\in \pi(g)\setminus \pi_g(Y)$ and $X_jg=b_j$}.
	\end{cases}
\end{align*}
It is clear that $h\in \overline{\Omega}(X,Y)$. First, we observe that $Xf=Xh$ and $|y(f_{\upharpoonright_Y})^{-1}|=|y(h_{\upharpoonright_Y})^{-1}|$ for all $y\in Y$.  
Therefore $(f,h)\in \mathcal{L}$ in $\overline{\Omega}(X,Y)$ by Theorem \ref{l-green}. 
Next, we see that $\pi(h)=\pi(g)$ and $\pi_h(Y)=\pi_g(Y)$. Therefore $(h,g) \in \mathcal{R}$ in $\overline{\Omega}(X,Y)$ by Theorem \ref{r-green}. Since $\mathcal{D} = \mathcal{L} \circ \mathcal{R}$, we conclude that $(f,g)\in \mathcal{D}$ in $\overline{\Omega}(X,Y)$.
\end{proof}

In order to characterize $\mathcal{J}$-relation on $\overline{\Omega}(X,Y)$, we need the following lemma.
\begin{lemma}\label{lemma-j-green}
Let $f,g\in \overline{\Omega}(X,Y)$. Then $f=hgh'$ for some $h,h'\in \overline{\Omega}(X,Y)$ if and only if 
\begin{enumerate}
\item[\rm(i)] $|Xf\setminus Y|\leq |Xg\setminus Y|$;
\item[\rm(ii)] there exists a partition $\{P_y\colon y\in Y\}$ of $Y$ such that \[|y(f_{\upharpoonright_Y})^{-1}|\geq \displaystyle\sum_{w\in P_y}|w(g_{\upharpoonright_Y})^{-1}|\] for all $y\in Y$.
\end{enumerate}
\end{lemma}

\begin{proof}
Suppose that $f=hgh'$ for some $h,h'\in \overline{\Omega}(X,Y)$.

\begin{enumerate}
\item[\rm(i)] Since  $\overline{\Omega}(X,Y) \subseteq \overline{T}(X,Y)$, we have $f,g,h,h' \in \overline{T}(X,Y)$ and so $|Xf\setminus Y|\leq |Xg\setminus Y|$ by \cite[Theorem 5]{hony11}.

\vspace{0.1cm}
\item[\rm(ii)] Note that $f_{\upharpoonright_Y}, g_{\upharpoonright_Y}, h_{\upharpoonright_Y}, h'_{\upharpoonright_Y} \in \Omega(Y)$ and $f_{\upharpoonright_Y}=h_{\upharpoonright_Y}g_{\upharpoonright_Y}h'_{\upharpoonright_Y}$. Hence, by \cite[Theorem 2.2]{koni-ac19}, there exists a partition $\{P_y\colon y\in Y\}$ of $Y$ such that $|y(f_{\upharpoonright_Y})^{-1}|\geq \sum_{w\in P_y}|w(g_{\upharpoonright_Y})^{-1}|$ for all $y\in Y$.
\end{enumerate}

\vspace{0.1cm}
Conversely, suppose that the given conditions hold. Since $Yf=Y$ and $Yg = Y$, we have $|Yf|= |Yg|$. Now, we recall (i). Then we get $|(Xf\setminus Y) \cup Y|\leq |(Xg\setminus Y) \cup Y|$ which simply gives $|Xf|\leq |Xg|$. Thus, by \cite[Theorem 5]{hony11}, there exist $\alpha, \gamma \in \overline{T}(X,Y)$ such that $f=\alpha g \gamma$.

\vspace{0.1cm}	

Now, we recall (ii). Then, \cite[Theorem 2.2]{koni-ac19}, there exist $\beta,\delta\in \Omega(Y)$ such that $f_{\upharpoonright_Y}=\beta g_{\upharpoonright_Y}\delta$. Let $x\in X\setminus Y$ such that $xf\in Y$. Write $xf = y$. Since $Yf=Y$, it follows that $Y\cap yf^{-1}\neq \emptyset$. Therefore we may choose an element $x' \in Y\cap yf^{-1}$. Define $h,h'\colon X \to X$ by
	\begin{align*}
		xh=
		\begin{cases}
			x\beta  & \text{if $x\in Y$},\\
			x'\beta &  \text{if $x\in X\setminus Y$ and $xf\in Y$},\\
			x\alpha  &  \text{if $x\in X\setminus Y$ and $xf\in X\setminus Y$}.
		\end{cases}\;
		\hspace{1.0cm}
		xh'=
		\begin{cases}
			x\delta  & \text{if $x\in Y$},\\
			x\gamma  &  \text{if $x\in X\setminus Y$}.
		\end{cases}
	\end{align*}
	
Clearly $h,h'\in \overline{\Omega}(X,Y)$. Finally, we show that $f = hgh'$. Let $x\in X$. Consider three possible cases.

\vspace{0.1cm}
\noindent \textbf{Case 1:} $x\in Y$. Then $x(hg) = (xh)g =(x\beta)g = (x\beta)g_{\upharpoonright_Y} = x(\beta g_{\upharpoonright_Y})$. Since $ x(\beta g_{\upharpoonright_Y}) \in Y$, we obtain
\[x(hgh')= (x(hg))h'= (x(\beta g_{\upharpoonright_Y}))h' = (x(\beta g_{\upharpoonright_Y}))\delta = x(\beta g_{\upharpoonright_Y}\delta)=xf_{\upharpoonright_Y}=xf.\]

\vspace{0.1cm}
\noindent \textbf{Case 2:} $x\in X\setminus Y$ and $xf\in Y$. Then $x(hg) =(xh)g =  (x'\beta)g = (x'\beta)g_{\upharpoonright_Y} = x'(\beta g_{\upharpoonright_Y})$. Since $ x'(\beta g_{\upharpoonright_Y}) \in Y$, we obtain
\[x(hgh')= (x(hg))h'= (x'(\beta g_{\upharpoonright_Y}))h' = (x'(\beta g_{\upharpoonright_Y}))\delta = x'(\beta g_{\upharpoonright_Y}\delta)=x'f_{\upharpoonright_Y}=x'f = xf.\]

\vspace{0.1cm}
\noindent \textbf{Case 3:} $x\in X\setminus Y$ and $xf\in X\setminus Y$. Then $x(hg) =(xh)g = (x\alpha)g = x(\alpha g)$. We now claim that $x(\alpha g) \in X\setminus Y$. We assume by contradiction that $x(\alpha g)\in Y$. Then $xf = x(\alpha g \gamma) = (x(\alpha g))\gamma \in Y$ which is a contradiction of the assumption $xf\in X\setminus Y$. Hence $x(\alpha g) \in X\setminus Y$. Using these facts, we obtain

\[x(hgh') = (x(hg)) h' = (x(\alpha g)) h' = (x(\alpha g)) \gamma = x(\alpha g \gamma) = xf.\]
 
\vspace{0.1cm}
Hence, from all three cases, we conclude that $f=hgh'$ as required. 
\end{proof}

\begin{theorem}\label{green-j-theorem}
	Let $f,g\in \overline{\Omega}(X,Y)$. Then $(f, g)\in \mathcal{J}$ in $ \overline{\Omega}(X,Y)$ if and only if $|Xf\setminus Y|=|Xg\setminus Y|$, and there exist partitions $\{P_y\colon y\in Y\}$ and $\{Q_y\colon y\in Y\}$ of $Y$ such that 
	 \[|y(f_{\upharpoonright_Y})^{-1}|\geq \sum_{w\in P_y}|w(g_{\upharpoonright_Y})^{-1}|	\hspace{0.4cm}\text{ and } \hspace{0.4cm} |y(g_{\upharpoonright_Y})^{-1}|\geq \sum_{w\in Q_y}|w(f_{\upharpoonright_Y})^{-1}|\]
	 for all $y\in Y$
	\end{theorem}

\begin{proof}
It follows immediately from Lemma \ref{lemma-j-green}.
\end{proof}

\vspace{0.1cm}
It is well-known that $\mathcal{D} \subseteq \mathcal{J}$ on every semigroup, and $\mathcal{D} = \mathcal{J}$ on $T(X)$. However, the following example shows that $\mathcal{D} = \mathcal{J}$ on $\overline{\Omega}(X,Y)$ is not always true.

\begin{example}
Let $Y=\mathbb{N}$ be the set of positive integers, and let $X=\mathbb{N}\cup \{0\}$. Define $f,g\colon X\to X$ by
	
\begin{equation*}
	xf =
	\begin{cases}
		1 & \text{if $x\in \{0,1\} \cup (2\mathbb{N}+1$)},\\
		\frac{x}{2}+1 & \text{if $x\in 2\mathbb{N}$}.
	\end{cases}
\hspace{0.4cm}
	xg =
	\begin{cases}
		1 & \text{if $x\in \{1\} \cup (4\mathbb{N}+1$)},\\
		2 & \text{if $x\in \{0\} \cup (4\mathbb{N}-1$)},\\
		\frac{x}{2}+2 & \text{if $x\in 2\mathbb{N}$}.
	\end{cases}
\end{equation*}		
	
Obviously $f,g\in \overline{\Omega}(X,Y)$. First, we show that $(f,g)\in \mathcal{J}$ in $\overline{\Omega}(X,Y)$. Observe that $Xf = Xg = Y$ and so $|Xf\setminus Y|=|Xg\setminus Y|$. Now, consider the partition $\{\{y\}\colon y\in Y\}$ of $Y$. Then, for all $y\in Y$, we obtain $|y(g_{\upharpoonright_Y})^{-1}| \ge \sum_{w\in \{y\}}|w(f_{\upharpoonright_Y})^{-1}|$. Next, consider the partition $\{\{1,2\}, \{y\}\colon y\in Y\setminus\{1,2\}\}$ of $Y$. Then, for all $y\in Y$, we obtain $|y(f_{\upharpoonright_Y})^{-1}| \ge \sum_{w\in P_y}|w(g_{\upharpoonright_Y})^{-1}|$.
Using these facts, we conclude from Theorem \ref{green-j-theorem} that $(f,g)\in \mathcal{J}$ on $\overline{\Omega}(X,Y)$.

\vspace{0.1cm}
Second, we assume by contradiction that $(f,g)\in \mathcal{D}$ in $\overline{\Omega}(X,Y)$. Then, by Theorem \ref{green-d-theorem}, there exists a bijection $\alpha \in S(Y)$ such that $|y(f_{\upharpoonright_Y})^{-1}|=|(y\alpha) (g_{\upharpoonright_Y})^{-1}|$ for all $y\in Y$. Now, we consider two possible cases of $1\alpha$.

\vspace{0.1cm}
\noindent \textbf{Case 1:} $1\alpha = 1$. Then for $y = 2\alpha^{-1}$, we get $|y(f_{\upharpoonright_Y})^{-1}| = 1$ and 
$|(y\alpha) (g_{\upharpoonright_Y})^{-1}| = |2 (g_{\upharpoonright_Y})^{-1}| = \infty$, a contradiction.

\vspace{0.1cm}
\noindent \textbf{Case 2:} $1\alpha \neq 1$. Then for $y = 1\alpha^{-1}$, we get $|y(f_{\upharpoonright_Y})^{-1}| = 1$ and 
$|(y\alpha) (g_{\upharpoonright_Y})^{-1}| = |1 (g_{\upharpoonright_Y})^{-1}| = \infty$, a contradiction.

\vspace{0.1cm}
Hence we conclude that$(f,g)\notin \mathcal{D}$ on $\overline{\Omega}(X,Y)$.
\end{example}

\vspace{0.1cm}
In the next result, we give a necessary and sufficient condition for $\mathcal{D} =\mathcal{J}$ on $\overline{\Omega}(X,Y)$.

\begin{theorem}\label{D=J}
We have $\mathcal{D}=\mathcal{J}$ on $\overline{\Omega}(X,Y)$ if and only if $Y$ is finite.
\end{theorem}
\begin{proof}
	Suppose that $\mathcal{D}=\mathcal{J}$ in $\overline{\Omega}(X,Y)$. On the contrary, assume that $Y$ is infinite. Then there exist two disjoint subsets $A, B \subseteq Y$ such that $|Y| =|A|=|B|=|Y\setminus (A\cup B)|$. Let $p, q \in Y$ be distinct elements. Note that $|Y\setminus (A\cup B)|=|Y\setminus \{p,q\}|$ and $|Y\setminus (A\cup B)|=|Y\setminus \{p\}|$. It follows that there exist bijections $\alpha \colon Y\setminus (A\cup B)\to Y\setminus\{p,q\}$ and $\beta \colon Y\setminus (A\cup B)\to Y\setminus\{p\}$. Define $f,g\colon X\to X$ by
	\begin{align*}
		xf=
		\begin{cases}
			p       &  \text{if $x\in A\cup (X\setminus Y)$},\\
			q       &  \text{if $x\in B$},\\
			x\alpha &  \text{otherwise}.
		\end{cases}\;
		\hspace{0.3cm}\text{ and } \hspace{0.3cm}
		xg=
		\begin{cases}
			p       &  \text{if $x\in A\cup B\cup (X\setminus Y)$},\\
			x\beta  &  \text{otherwise}.
		\end{cases}
	\end{align*}
It is clear that  $f,g\in \overline{\Omega}(X,Y)$. First, we show that $(f,g)\in \mathcal{J}$ in $\overline{\Omega}(X,Y)$. Observe that $|Xf\setminus Y|=|Xg\setminus Y|$. Now, consider the partition $\{\{y\}\colon y\in Y\}$ of $Y$. Then, for all $y\in Y$, we obtain $|y(f_{\upharpoonright_Y})^{-1}| \ge \sum_{w\in \{y\}}|w(g_{\upharpoonright_Y})^{-1}|$.
 Next, consider the partition $\{\{p,q\}, \{y\}\colon y\in Y\setminus\{p,q\}\}$ of $Y$. Then, for all $y\in Y$, we obtain $|y(g_{\upharpoonright_Y})^{-1}| \ge \sum_{w\in Q_y}|w(f_{\upharpoonright_Y})^{-1}|$. Using these facts, we conclude from Theorem \ref{green-j-theorem} that $(f,g)\in \mathcal{J}$ on $\overline{\Omega}(X,Y)$.

 \vspace{0.1cm}
 Second, observe that the number of blocks of infinite cardinalities in the partitions $\pi(f_{\upharpoonright_Y})$ and $\pi(g_{\upharpoonright_Y})$ of $Y$ are two and one, respectively. Therefore there does not exist any bijection $\alpha \in S(Y)$ such that $|y(f_{\upharpoonright_Y})^{-1}|=|(y\alpha) (g_{\upharpoonright_Y})^{-1}|$ for all $y\in Y$. It follows from Theorem \ref{green-d-theorem} that $(f,g)\notin \mathcal{D}$ in $\overline{\Omega}(X,Y)$ which is a contradiction of the assumption  $\mathcal{D}=\mathcal{J}$ on $\overline{\Omega}(X,Y)$. Hence $Y$ is finite.
	
\vspace{0.1cm}
Conversely, suppose that $Y$ is finite. Then $\overline{\Omega}(X,Y) = \overline{S}(X,Y)$. Hence, by \cite[Theorem 3.1(4)]{somma-tjm21}, we have $\mathcal{D}=\mathcal{J}$ on $\overline{\Omega}(X,Y)$  
\end{proof}

\vspace{0.1cm}
Because of Example 2.3 of \cite{koni-ac19}, we remark that $\mathcal{D}=\mathcal{J}$ on $\Omega(X)$ if and only if $X$ is finite. Using this, we have the following corollary of Theorem \ref{D=J}.

\begin{corollary}
We have $\mathcal{D}=\mathcal{J}$ on $\overline{\Omega}(X,Y)$ if and only if $\mathcal{D}=\mathcal{J}$ on $\Omega(Y)$.
\end{corollary}


\section{Ideals of $\overline{\Omega}(X,Y)$}

In this section, we present a characterization of ideals of $\overline{\Omega}(X, Y)$. Moreover, we determine the kernel of $\overline{\Omega}(X, Y)$. To prove the main results of this section, we need the next five lemmas.

\vspace{0.1cm}
 For an infinite set $A$ and a map $\alpha \in \Omega(A)$, let
 \[N(\alpha)=\{a\in A\colon |a\alpha^{-1}|<|A|\}\] and $n(\alpha)=|N(\alpha)|$. It is obvious that  $0\leq n(\alpha)\leq |A|$ for all $\alpha \in \Omega(A)$.

\begin{lemma}\label{n(alpha)}
If $A$ is an infinite set and $\alpha,\beta \in \Omega(A)$, then $n(\alpha \beta)\leq \min\{n(\alpha),n(\beta)\}$.
\end{lemma}
\begin{proof}
First, we show that $n(\alpha \beta)\leq n(\beta)$. For this, let $a\in N(\alpha \beta)$. Then $|a(\alpha \beta)^{-1}|<|A|$ and so $|a\beta^{-1}|\leq |(a\beta^{-1})\alpha^{-1}| = |a(\alpha \beta)^{-1}|<|A|$. Therefore $a\in N(\beta)$ and so $N(\alpha \beta)\subseteq N(\beta)$. Hence $n(\alpha \beta)\leq n(\beta)$.
	
\vspace{0.1cm}	
Next, we show that $n(\alpha \beta)\leq n(\alpha)$. For this, let $a\in N(\alpha \beta)$. Then $|a(\alpha \beta)^{-1}|<|A|$. Let $x\in a\beta^{-1}$. Then $|x\alpha^{-1}|\leq |(a\beta^{-1})\alpha^{-1}| =  |a(\alpha \beta)^{-1}| <|A|$. This implies that $x \in N(\alpha)$ and so $a = x\beta \in N(\alpha)\beta$. Therefore $N(\alpha \beta)\subseteq N(\alpha)\beta$ and so $n(\alpha \beta)\leq |N(\alpha)\beta|$. Since $|N(\alpha)\beta|\le n(\alpha)$, we subsequently have $n(\alpha \beta)\leq n(\alpha)$. Thus $n(\alpha \beta)\leq \min\{n(\alpha),n(\beta)\}$.	
\end{proof}

\vspace{0.1cm}
For two cardinal numbers $s$ and $t$ with $0\leq s\leq |Y|$ and $0\leq t\leq |X\setminus Y|$, let 
\[J(s,t)=\{f\in \overline{\Omega}(X,Y)\colon n(f_{\upharpoonright_Y})\leq s\mbox{ and }|Xf\setminus Y|\leq t\}.\]

\begin{lemma}\label{ideal-J(s,t)}
Let $s$ and $t$ be cardinal numbers such that $0\leq s\leq |Y|$ and $0\leq t\leq |X\setminus Y|$. If $Y$ is an infinite subset of $X$, then $J(s,t)$ is an ideal of $\overline{\Omega}(X,Y)$.
\end{lemma}
\begin{proof}
First, we show that $J(s,t) \neq \emptyset$. For this, consider a map $f\in T(X)$ such that $Xf = Y$ and $|y(f_{\upharpoonright_Y})^{-1}| = |Y|$ for all $y\in Y$. Obviously $f\in \overline{\Omega}(X,Y)$. Also, it is clear that $n(f_{\upharpoonright_Y}) = 0$ and $|Xf\setminus Y| = 0$. Therefore $f\in J(s, t)$ and so $J(s,t) \neq \emptyset$.

\vspace{0.1cm}
To show that $J(s,t)$ is an ideal of $\overline{\Omega}(X,Y)$, let $f\in J(s,t)$ and $g,h\in \overline{\Omega}(X,Y)$. Then, by definition of $J(s,t)$, we have $n(f_{\upharpoonright_Y})\leq s$ and $|Xf\setminus Y|\leq t$. Note that 
    $(gfh)_{\upharpoonright_Y}=g_{\upharpoonright_Y}f_{\upharpoonright_Y}h_{\upharpoonright_Y}$. Therefore, by Lemma \ref{n(alpha)}, we obtain \[n((gfh)_{\upharpoonright_Y})=n(g_{\upharpoonright_Y}f_{\upharpoonright_Y}h_{\upharpoonright_Y})\leq n(f_{\upharpoonright_Y}h_{\upharpoonright_Y})\leq n(f_{\upharpoonright_Y})\leq s.\]
    
Since $Yh = Y$, we have $X(fh)\setminus Y \subseteq (Xf\setminus Y)h$. Also, since $Xg \subseteq X$, we have
$X(gfh) \subseteq X(fh)$. Using these facts, we obtain  \[|X(gfh)\setminus Y|\leq|X(fh)\setminus Y|\leq |(Xf\setminus Y)h|\leq |Xf\setminus Y|\leq t.\]
Hence, by definition of $J(s,t)$, we have $gfh\in J(s,t)$ as required. 
\end{proof}

\begin{lemma}
Let $s$ and $t$ be cardinal numbers such that $0\leq s\leq |Y|$ and $0\leq t\leq |X\setminus Y|$. Then \[\bigcup_{f\in J(s, t)} J_f = J(s, t).\]
\end{lemma}

\begin{proof}
Write $\bigcup_{f\in J(s, t)} J_f = J$. It is obvious that $J(s, t) \subseteq J$. To show the reverse inclusion, let $g\in J$. Then $g\in J_f$ for some $f\in J(s, t)$.  Therefore, by Theorem \ref{green-j-theorem}, we have $|Xg\setminus Y| = |Xf\setminus Y|\le t$. Also, by Theorem \ref{green-j-theorem}, there exist partitions $\{P_y\colon y\in Y\}$ and $\{Q_y\colon y\in Y\}$ of $Y$ such that, for all $y\in Y$,
\begin{equation}\label{J-class=J(s,t)}
|y(f_{\upharpoonright_Y})^{-1}|\geq \sum_{w\in P_y}|w(g_{\upharpoonright_Y})^{-1}|   \hspace{0.5cm} \text{ and }\hspace{0.5cm}|y(g_{\upharpoonright_Y})^{-1}|\geq \sum_{v\in Q_y}|v(f_{\upharpoonright_Y})^{-1}|.
\end{equation}

Now, we show that $n(g_{\upharpoonright_Y}) =n(f_{\upharpoonright_Y})$. For this, let $x\in N(f_{\upharpoonright_Y})$. Then $|x(f_{\upharpoonright_Y})^{-1}| < |Y|$. Let $w\in P_x$. Then, from the first inequality of equation (\ref{J-class=J(s,t)}), we get
$|w(g_{\upharpoonright_Y})^{-1}| < |Y|$ and so $w\in N(g_{\upharpoonright_Y})$. It follows that $P_x\subseteq N(g_{\upharpoonright_Y})$. Since $|P_x| \ge 1$, we have $|N(f_{\upharpoonright_Y})| \le |N(g_{\upharpoonright_Y})|$ and subsequently $n(f_{\upharpoonright_Y}) \le n(g_{\upharpoonright_Y})$.

\vspace{0.1cm}
Similarly, by using the second inequality of equation (\ref{J-class=J(s,t)}), we can obtain $n(g_{\upharpoonright_Y}) \le n(f_{\upharpoonright_Y})$. Thus $n(g_{\upharpoonright_Y}) =n(f_{\upharpoonright_Y})$. Since $n(f_{\upharpoonright_Y})\leq s$, it follows that $n(g_{\upharpoonright_Y})\leq s$.  Hence, by definition of $J(s,t)$, we have $g\in J(s,t)$ and so $J \subseteq J(s,t)$ as required.	
\end{proof}

\vspace{0.1cm}

\vspace{0.1cm}
For a nonempty subset $F$ of $\overline{\Omega}(X,Y)$, let
\begin{equation*}
	\begin{split}
		J(F)&=\big\{f\in \overline{\Omega}(X,Y)\colon \text{ there exist } g\in F \text{ and a partition } \{P_y\colon y\in Y\} \text{ of } Y \text{ such that }\\
		&\quad \hspace{0.4cm} |Xf\setminus Y|\leq |Xg\setminus Y|\text{ and } |y(f_{\upharpoonright_Y})^{-1}|\geq \sum_{w\in P_y}|w(g_{\upharpoonright_Y})^{-1}| \text{ for all } y\in Y \big\}.
	\end{split}
\end{equation*}

\vspace{0.2cm}
It is clear that  $F\subseteq J(F)$. 

\vspace{0.1cm}
\begin{lemma}\label{J(Z)-ideal}
If $F$ is a nonempty subset of $\overline{\Omega}(X,Y)$, then $J(F)$ is an ideal of $\overline{\Omega}(X,Y)$.
\end{lemma}
\begin{proof}
Note that $F\subseteq J(F)$. Since $F \neq \emptyset$, we have $J(F)\neq \emptyset$. To show that $J(F)$ is an ideal of $\overline{\Omega}(X,Y)$, let $f\in J(F)$ and $h,h'\in \overline{\Omega}(X,Y)$. Then, by definition of $J(F)$, there exist $g\in F$ and a partition $\{P_y\colon y\in Y\}$ of $Y$ such that $|Xf\setminus Y|\leq |Xg\setminus Y|$ and for all  $y\in Y$ 
\begin{equation}\label{ineq-JF}
	|y(f_{\upharpoonright_Y})^{-1}|\geq \sum_{w\in P_y}|w(g_{\upharpoonright_Y})^{-1}|. 
\end{equation}

First, note that $Yh' = Y$ and $X(hfh') \subseteq X(fh')$. Therefore we obtain
\[|X(hfh')\setminus Y|\leq|X(fh')\setminus Yh'|\le |(Xf\setminus Y)h'|\leq |Xf\setminus Y| \leq |Xg\setminus Y|.\]	 

Second, let $y\in Y$. Since  $f_{\upharpoonright_Y},h_{\upharpoonright_Y}, h'_{\upharpoonright_Y}\in \Omega(Y)$, we obtain 
\begin{align*}
|y(h_{\upharpoonright_Y}f_{\upharpoonright_Y}h'_{\upharpoonright_Y})^{-1}|
&=|(y(h_{\upharpoonright_Y}')^{-1}(f_{\upharpoonright_Y})^{-1})(h_{\upharpoonright_Y})^{-1}|
\\&\geq|y(h_{\upharpoonright_Y}')^{-1}(f_{\upharpoonright_Y})^{-1}|\\
&=\sum_{x\in y(h'_{\upharpoonright_Y})^{-1}} |x(f_{\upharpoonright_Y})^{-1}|\\
&\geq \sum_{x\in y(h'_{\upharpoonright_Y})^{-1}}\sum _{w\in P_x}|w(g_{\upharpoonright_Y})^{-1}| \hspace{1.0cm}\text{ by equation } (\ref{ineq-JF})\\
&= \sum _{w\in Q_y}|w(g_{\upharpoonright_Y})^{-1}|, \hspace{0.5cm}\text{ where } \hspace{0.2cm} Q_y = \bigcup_{x\in y(h'_{\upharpoonright_Y})^{-1}} P_x.
\end{align*}
Observe that $\{Q_y\colon y\in Y\}$ is a partition of $Y$. 
Thus, by definition of $J(F)$, we have $hfh'\in J(F)$ as required. 
\end{proof}

\vspace{0.1cm}

\begin{lemma}\label{J(Z)=Z}
If $F$ is an ideal of $\overline{\Omega}(X,Y)$, then $F = J(F)$.
\end{lemma}
\begin{proof}
From definition of $J(F)$, we have $F\subseteq J(F)$. To show the reverse inclusion, let $f\in J(F)$. Then, by definition of $J(F)$, there exist $g\in F$ and a partition $\{P_y\colon y\in Y\}$ of $Y$ such that $|Xf\setminus Y|\leq |Xg\setminus Y|$ and 
	$|y(f_{\upharpoonright_Y})^{-1}|\geq \sum_{w\in P_y}|w(g_{\upharpoonright_Y})^{-1}|$ for all $y\in Y$. Then, by Lemma \ref{lemma-j-green}, there exist $h,h'\in \overline{\Omega}(X,Y)$ such that $f=hgh'$. Since $F$ is an ideal of $\overline{\Omega}(X,Y)$, it follows that $f=hgh'\in F$ and so $J(F)\subseteq F$ as required.
\end{proof}

\vspace{0.1cm}
Combining Lemmas \ref{J(Z)-ideal} and \ref{J(Z)=Z}, we obtain a characterization of ideals of $\overline{\Omega}(X,Y)$ as follows. 

\begin{theorem}\label{ideal-OXY}
An ideal of $\overline{\Omega}(X,Y)$ is precisely of the form $J(F)$ for some nonempty subset $F$ of $\overline{\Omega}(X,Y)$.
\end{theorem}

\vspace{0.1cm}
If $Y=X$, then $\overline{\Omega}(X,Y) = \Omega(X)$. The following corollary is an immediate consequence of Theorem \ref{ideal-OXY}.

\begin{corollary}\label{ideal-Omega(X)}
Let $X$ be an infinite set. Then an ideal of $\Omega(X)$ is precisely 

\begin{equation*}
	\begin{split}
		&\big\{f\in \Omega(X)\colon \text{ there exist } g\in F \text{ and a partition } \{P_x\colon x\in X\} \text{ of }X \text{ such that }\\
		&\quad \hspace{1.8cm} |xf^{-1}|\geq \sum_{w\in P_x}|w g^{-1}|
		\text{ for all } x\in X \big\}
	\end{split}
\end{equation*}
for some nonempty subset $F$ of $\Omega(X)$.
\end{corollary}

\vspace{0.1cm}
In the next result, we determine the kernel of $\overline{\Omega}(X,Y)$.

\begin{theorem}\label{ker-OXY}
Let $Y$ be an infinite subset of $X$. Then $K(\overline{\Omega}(X,Y)) = J(0,0)$.
\end{theorem}
\begin{proof}
From Lemma \ref{ideal-J(s,t)}, the subset $J(0,0)$ is an ideal of $\overline{\Omega}(X,Y)$. To show that $J(0,0)$ is minimal, let $I$ be an ideal of $\overline{\Omega}(X,Y)$ that is contained in $J(0,0)$. Let $f\in J(0,0)$. Then, by definition of $J(0,0)$, we have $|Xf\setminus Y|=0$ and $n(f_{\upharpoonright_Y})=0$. 

\vspace{0.1cm}
Let $g\in I$. Obviously $|Xg\setminus Y|\ge 0$. Since $|Xf\setminus Y|=0$, it is clear that $|Xf\setminus Y|\leq |Xg\setminus Y|$. Let $y\in Y$. Since $n(f_{\upharpoonright_Y})=0$, we have $|y(f_{\upharpoonright_Y})^{-1}|=|Y|$ and so  
\[|y(f_{\upharpoonright_Y})^{-1}|=|Y|\geq \sum_{w\in \{y\}}|w(g_{\upharpoonright_Y})^{-1}|\]
for the partition $\{ \{y\}\colon y\in Y\}$ of $Y$. It follows from definition of $J(I)$ that $f\in J(I)$ and so $f\in I$ by Lemma \ref{J(Z)=Z}. Therefore $J(0,0)\subseteq I$ and hence $J(0,0) = I$ as required.
\end{proof}

\vspace{0.1cm}

If $Y=X$, then $\overline{\Omega}(X,Y) = \Omega(X)$. The following corollary is an immediate consequence of Theorem \ref{ker-OXY}, which was first proved by Konieczny \cite[Theorem 3.2(1)]{koni-ac19} in $2019$.

\begin{corollary}
If $X$ is an infinite set, then
$K(\Omega(X))=\{\alpha \in \Omega(X)\colon n(\alpha)=0\}$.
\end{corollary}



\end{document}